\documentclass[12pt]{article}
\usepackage{amssymb,amsmath,amsthm,mathrsfs}
\usepackage[usenames]{color}
\usepackage{hyperref}
\usepackage[all]{xy}
\usepackage{xypic}
\usepackage{graphicx}
\usepackage{amscd}
\usepackage{enumerate}
\usepackage[shortlabels]{enumitem}

\usepackage[top=2cm, left=2cm, right=2cm, bottom=2cm]{geometry}
\newcommand{\Rami}[1]{{{#1}}}

\newtheorem{theorem}{Theorem}[section]
\newtheorem{proposition}[theorem]{Proposition}
\newtheorem{corollary}[theorem]{Corollary}

\newtheorem{lemma}[theorem]{Lemma}
\newtheorem{definition}[theorem]{Definition}
\newtheorem{notation}[theorem]{Notation}

\newtheorem{remark}[theorem]{Remark}

\newtheorem{introtheorem}{Theorem}

\newtheorem{introcorollary}[introtheorem]{Corollary}

\DeclareMathOperator{\val}{val}

\DeclareMathOperator{\Spec}{Spec}
\DeclareMathOperator{\spec}{Spec}
\DeclareMathOperator{\GL}{GL}

\DeclareMathOperator{\Gal}{Gal}

\DeclareMathOperator{\Frac}{Frac}
\DeclareMathOperator{\Id}{Id}

\newcommand{\fX}{\mathfrak{X}}
\newcommand{\fY}{\mathfrak{Y}}
\newcommand{\fZ}{\mathfrak{Z}}

\newcommand{\cF}{\mathcal{F}}
\newcommand{\cS}{\mathcal{S}}
\newcommand{\bZ}{\mathbb{Z}}

\usepackage{amsmath,amssymb}

\title{Pointwise surjective presentations of stacks}
\author{Avraham Aizenbud and Nir Avni}
\begin{document}

\maketitle

\begin{abstract} We show that any stack $\mathfrak{X}$ of finite type over a Noetherian scheme has a presentation $X \rightarrow \mathfrak{X}$ by a scheme of finite type such that $X(F) \rightarrow \mathfrak{X}(F)$ is onto, for every finite or real closed field $F$. Under some additional conditions on $\mathfrak{X}$, we show the same for all perfect fields. We prove similar results for (some) Henselian rings.

We give two applications of the main result. One is to counting isomorphism classes of stacks over the rings $\mathbb{Z}/p^n$; the other is about the relation between real algebraic and Nash stacks.
\end{abstract}

\tableofcontents
\section{Introduction} \label{sec:introduction}

Let $\mathfrak{X}$ be an algebraic stack. By definition, there exist a scheme $X_0$ and a submersive smooth map $X \rightarrow \mathfrak{X}$; such a map is called a presentation of $\mathfrak{X}$. Let $X_1=X_0 \times_{\mathfrak{X}} X_0$ be the fiber product. Note that, in general, $X_1$ is an algebraic space. The two projections $s,t:X_1 \rightarrow X_0$, together with the diagonal $\Delta:X_0 \rightarrow X_1$ and the composition map $c:X_1 \times_{s,X_0,t} X_1 \rightarrow X_1$ give a structure of a groupoid object $(X_0,X_1,s,t,\Delta,c)$ in algebraic spaces: here $X_0$ is the space of objects, $X_1$ is the space of morphisms, $s$ and $t$ are the source and target maps, $\Delta$ is the identity map, and $c$ is the composition map.

The groupoid object $(X_0,X_1,s,t,\Delta,c)$ is closely related to $\mathfrak{X}$. In particular, for any field $F$, there is a natural and fully faithful functor from $(X_0(F),X_1(F),s,t,\Delta,c)$ to $\mathfrak{X}(F)$. For algebraically closed fields, this functor is an  equivalence of groupoids. However, this is false in general: taking $\mathfrak{X}$ to be the classifying space of the group $C_2$ and $X_0$ to be a point, we have that $X_1$ is a pair of points and, for every field $F$, the groupoid  $(X_0(F),X_1(F),s,t,\Delta,c)$ has only one object, whereas the isomorphism classes in $\mathfrak{X}(F)$ are in bijection with the square class group of $F$.

In this paper we show that every algebraic stack has a presentation such that the above functor is an equivalence of groupoids, for any finite or real-closed field $F$. We also show that, under some condition on the stack $\mathfrak{X}$, there is a presentation such that the above functor is an equivalence of groupoids, for any perfect field $F$. The results also extend to Henselian rings with residue fields of the above form.

We give two applications of the main result. The first is to the study of the sequence $|\pi_0(\mathfrak{X}(\mathbb{Z} / n))|$, where $\mathfrak{X}$ is a stack defined over $\mathbb{Z}$ and $\pi_0(\mathfrak{X}(\mathbb{Z} /n))$ is the set of isomorphism classes of $\mathfrak{X}(\mathbb{Z} / n)$. The second is to show that, for any algebraic stack $\mathfrak{X}$ defined over $\mathbb{R}$, the groupoid $\mathfrak{X}(\mathbb{R})$ has a structure of a Nash groupoid.

\subsection{Formulation of the main results}

We fix a Noetherian scheme $S$. All the schemes/group schemes/algebraic spaces/algebraic stacks we will consider will be of finite type over $S$ unless stated otherwise.

\begin{definition} Let $\pi :X\rightarrow \frak{X}$ be presentation of an algebraic stack and let $T \in Sch_{/S}$ be a scheme.
\begin{enumerate}
\item A $T$-point $T \rightarrow \frak{X}$ is $\pi$-liftable if it factors through some map $T \rightarrow X$ (up to isomorphism).
\item We say that $\pi$ is $T$-onto if every $T$ point of $\frak{X}$ is $\pi$-liftable.
\item Let $\mathcal{S} \subset Sch_{/S}$ be a full subcategory of the overcategory of $S$. We say that $\pi$ is $\mathcal{S}$-onto if it is $T$-onto, for every object of $\mathcal{S}$.
\end{enumerate}
\end{definition}

\begin{definition}$ $
We denote:
\begin{itemize}
\item by $\mathcal{F} \subset Sch_{/S}$ the category of spectra of fields,
\item by $\mathcal{F}_{\mathrm{perf}} \subset \mathcal{F}$ the category of spectra of perfect fields,
\item by $\mathcal{F}_f \subset \mathcal{F}$ the category of spectra of finite fields,
\item by $\mathcal{F}_r \subset \mathcal{F}$ the category of spectra of real closed fileds,
\item by $\mathcal{H} \subset Sch_{/S}$ the category of  Henselian schemes (i.e. spectra of Henselian local rings),
\item and by $\mathcal{H}_{\mathrm{perf}}, \mathcal{H}_{f}, \mathcal{H}_{r}$ the categories of Henselian schemes whose closed points are in $\mathcal{F}_{\mathrm{perf}}, \mathcal{F}_{f}, \mathcal{F}_{r}$ respectively.
\end{itemize}
\end{definition}


\begin{definition}[{cf. \cite[Definition 1.1.8]{DG}}]
$ $
\begin{itemize}
\item Let $\frak{X}$ be a stack, let $T$ be a scheme, and let $x\in \frak{X}(T)$ be a $T$-point. For any $T$-scheme $R\to T$, define $Aut(x)(R):=Aut(y)$ where $y$ is the $R$-point of $\fX$ defined by the composition $R\to T\to \fX$.
\item
We say that $\fX$ is QCA if for any separably closed field $F$ and for any $F$-point $x$ of $\fX$, the functor $Aut(x)$ is represented by a linear algebraic group
\end{itemize}

\end{definition}

The main result in this paper is the following:

\begin{introtheorem}[See \S\ref{sec:lift}] \label{thm:lift} Suppose $\frak{X}$ is a finite type stack over $S$. Then
 \begin{enumerate}
\item There is a presentation $X \rightarrow \frak{X}$ which is $\mathcal{H}_f$-onto and $\mathcal H_r$-onto.
\item 
If $\fX$ is QCA, then there is a presentation $X \rightarrow \frak{X}$ which is $\mathcal{H}_{\mathrm{perf}}$-onto.
\end{enumerate}
\end{introtheorem}

We will deduce this theorem from a statment that any surjective presentation of a stack satisfies some weaker condition that we define now:

\begin{definition} Let $\mathcal{S} \subset \mathcal{F}$ and let $\phi:\frak{X} \rightarrow \frak{Y}$ be a morphism of stacks. We say that $\phi$ is $(n,\mathcal{S})$-almost onto if, for every $\spec F \in \mathcal{S}$ and any $F$-point $u:\spec F \rightarrow \frak{Y}$, there is a separable field extension $E/F$ of degree at most $n$ such that the composition $\Spec E \rightarrow \Spec F \rightarrow \frak{Y}$ factors through $\phi$ (up to isomorphism).
\end{definition}

\begin{introtheorem}[See \S\ref{sec:key.lemma}] \label{lem:key} Let $\pi : X \rightarrow \frak{X}$ be a surjective map of finite type between a scheme $X$ and a stack $\frak{X}$. Then there exists $n$ such that
\begin{enumerate}
\item $\pi$ is $(n,\mathcal{F}_f)$-almost onto.
\item If $\frak{X}$ is QCA, then $\pi$ is $(n,\mathcal{F})$-almost onto.
\end{enumerate}
\end{introtheorem}

\subsection{Applications}

Theorem \ref{thm:lift} implies:

\begin{introcorollary} Let $\mathfrak{X}$ be a stack of finite type over $\Spec \mathbb{Z}$. Then \begin{enumerate}
\item For every $p$, the power series $\sum | \pi_0(\mathfrak{X}(\mathbb{Z} / p^n))| t^n$ is a rational function of $t$.
\item The Dirichlet series $\sum_n | \pi_0(\mathfrak{X}(\mathbb{Z} / n) | n^{-s}$ has a rational abscissa of convergence.
\end{enumerate}
\end{introcorollary}

\begin{proof} By Theorem \ref{thm:lift}, there is a presentation $X \rightarrow \mathfrak{X}$ such that, for every finite ring $R$, the map $X(R) \rightarrow \pi_0(\mathfrak{X}(R))$ is onto. Again, by Theorem \ref{thm:lift}, there is a presentations $Y \rightarrow X \times_\mathfrak{X} X$ such that, for every finite ring $R$, the map $Y(R) \rightarrow (X \times_\mathfrak{X} X)(R)$ is onto. Denote the composition $Y \rightarrow X \times_\mathfrak{X} X \rightarrow X \times X$ by $f$. Then $f(Y(R))$ is an equivalence relation on the set $X(R)$ and $| \pi_0(\mathfrak{X}(R))|=| X(R) / f(Y(R)) |$, for every finite ring $R$. \cite[Theorems 1.3 and 1.4]{HMRC} imply the corollary.
\end{proof}

Let $\mathfrak{X}$ be a smooth algebraic stack of finite type defined over $\Spec \mathbb{R}$. In \cite[Appendix A]{Sak}, Sakellaridis defines a stack $\mathfrak{X} ^{Nash}$ on the site of Nash manifolds and asks whether it is always a Nash stack (i.e., is there a smooth presentation of $\mathfrak{X} ^{Nash}$ by a Nash manifold). A criterion for $\mathfrak{X}^{Nash}$ being a Nash stack is given in \cite[Proposition A.1.4]{Sak}. Using this criterion and Theorem \ref{thm:lift}, we get

\begin{introcorollary} Let $\mathfrak{X}$ be a smooth algebraic stack of finite type defined over $\Spec \mathbb{R}$ in the sense of \cite[\S2.3]{Sak}\footnote{The definition in \cite[\S2.3]{Sak} is slightly more restrictive, though we believe the result is true without the restriction.} and assume that the diagonal map $\mathfrak{X} \rightarrow \mathfrak{X} \times \mathfrak{X}$ is schematic and separated. Then $\mathfrak{X} ^{Nash}$ is a Nash stack.
\end{introcorollary}

\subsection{Sketch of the proof of the main results}

We prove Theorem \ref{thm:lift} using Theorem \ref{lem:key}. Unlike Theorem \ref{thm:lift}, Theorem \ref{lem:key} can be proved by stratifying $\mathfrak{X}$ and proving the theorem for each stratum.

We prove Theorem \ref{lem:key} by analyzing a sequence of special cases:
\begin{enumerate}[{Case} 1.]
\item $\frak{X}$ is a scheme and $\pi : X \rightarrow \frak{X}$ is quasi-finite. In this case the number $n$ is obtained from the degree of the fibers of $\pi$.
\item\label{ca:2} $\frak{X}$ is a scheme. This case follows from the previous one. This case allow us, for a fixed stack $\fX$, to deduce the theorem for arbitrary surjection $\pi : X \rightarrow \frak{X}$ from knowing it for one surjection $\pi : X \rightarrow \frak{X}$
\item\label{ca:3} $\frak{X}$ is an algebraic space. We may assume that $\pi$ is an etale presentation of $\frak{X}$. Then  $\frak{X}$ can be viewed as a quotient of $X$ by an etale equivalence relation. The number $n$ comes from the size of the equivalence clases.
\item $\fX=BG$ for an algebraic group $G$.  The statement can be reformulated as a statement about Galois cohomology. In the case of a finite field we use Lang's theorem, and the number $n$ comes from the number of components of $G$. In the case of a QCA stack, the group $G$ is linear and thus can be embedded into $GL_n$. Here the proof is based on Hilbert 90 and on Case \ref{ca:2} applied to the map $GL_n \to GL_n/G$
\item $\fX=BG$, for a group scheme $G$ over a base scheme $Y$.
Over the generic point it looks like the previous case. From this we deduce the theorem for an open dense subset in $X$ and proceed by Noetherian induction.
\item $\fX$ is a gerbe over an algebraic space. This follows from the previous case and Case \ref{ca:3}.
\item  The general case. This follows from the previous case using the fact that any stack can be stratified into gerbes.
\end{enumerate}
In order to deduce Theorem \ref{thm:lift}, we introduce a construction that starts with an almost onto presentation $\pi:X \to \fX$ and gives an onto one. This is done in the following steps:

\begin{enumerate}[{Step} 1.]
\item Given two $S$-schemes $X$ and $Y$, one can define the internal hom $X^Y$ over $S$ as a pre-sheaf on the category of $S$-schemes. This  pre-sheaf is often not representable, but under some restrictive conditions on $X,Y$ it is  representable by an algebraic space and, under more restrictive conditions, by a scheme.
\item More generally given two diagrams of  $S$-schemes $D_1$ and $D_2$ of the same shape, we define internal hom ${D_1}^{D_2}$ over $S$ as a pre-sheaf on the category of $S$-schemes. Again it will be representable under some conditions
\item Given a presentation $\pi:X \to \fX$ one can construct a simplicial scheme  $[\pi]_{\bullet}$ called the Cech nerve of $\pi$ by taking the fiber powers of $X$ over $X$. The diagram $[\pi]_{\bullet}$ is an infinite diagram, but, since  $[\pi]_{\bullet}$ is coskeletal, its behavior is determined by a finite sub-diagram (the first three levels).
\item Given a presentation $\pi:X \to \fX$  and an etale covering $\tau: S'\to S$, we construct a new presentation
 $$f_{\pi,\tau}:{[\pi]_{\bullet}}^{[\tau]_{\bullet}}\to \fX.$$
This presentation tends to be more onto than $\pi$. For example if $\tau: \spec(E)\to \spec(F)$  is a finite field extension, then if a composition $\spec(E)\to \spec(F) \to \fX$ is  $\pi$- liftable then the map $\spec(F) \to \fX$ is $f_{\pi,\tau}$-liftable.
\item For an integer $n$ we construct an etale map $\tau_n:\mathbb U_n'\to \mathbb U_n$ that packages all separable field extensions of degree $\leq n$. Namely, for any separable field extension $E/F$ of degree $\leq n$, there is an $F$ point  of  $\mathbb U_n$ whose fiber is $\Spec E$.
\item We combine the last two steps. Namely, given an $n$-onto presentation $\pi:X \to \fX$  we consider the presentation $\pi_n:X\times \mathbb U_n \to \fX$  and then the presentation

$$f_{\pi_n,\tau_n}:{[\pi_n]_\bullet}^{[\tau_n]_\bullet}\to \fX.$$
\item Denote $X_n:={[\pi_n]_\bullet}^{[\tau_n]_\bullet}$. The obtained presentation $X_n \to \fX$ is an onto one, however it might be that $X_n$ is not a scheme, but an algebraic space. In order to complete the construction we present $X_n$ by an affine scheme $Y$ and repeat the construction for the presentation $Y\to X_n$. The composition of the obtained presentation with $f_{\pi_n,\tau_n}$ is an onto presentation by a scheme.
\end{enumerate}

%
%

\subsection{Acknowledgments}
We thank Shahar Carmeli, Raf Cluckers, and Ofer Gabber for fruitful discussions. We thank Angelo Vistoli for answering a question of ours on MathOverFlow, proving Lemma \ref{lem:left.lift}. A.A. was partially supported by ISF grants 687/13, and  grant 249/17. N.A. was partially supported by NSF grant  DMS-1902041. Both authors was  partially  supported by BSF grants  2012247 and 2018201.

\section{Almost onto presentations of stacks (Proof of Theorem \ref{lem:key})} \label{sec:key.lemma}
In some stages of the proof we will stratify of the stack $\fX$ and prove the claims for the strarta. The following allows us to do that.
\begin{lemma}\label{lem:strat}
Let $\fX:=\bigcup \fX_i$ be a finite stratification of an algebraic stack (See \cite[97.28]{SP}). Then for any field $F$,  any $F$-point $x:\spec(F)\to \fX$  factors through one of the $\fX_i$ (up to an isomorphism).
\end{lemma}
\begin{proof}
Consider the base-change of $x$ to $\bigsqcup \frak{X}_i$. This is  a non-trivial closed immersion $\Spec F \times_\frak{X} \sqcup \frak{X}_i \rightarrow \Spec F$, so it is invertible. This implies that $x$ factors through $\Spec F \rightarrow \frak{X}_i$, for some $i$.
\end{proof}

%


The following will allow us to replace an arbitrary surjective morphism by a quasi-finite one:
\begin{lemma} \label{lem:q.f.section} Let $\pi :X \rightarrow Y$ be a surjective morphism of finite type between Noetherian schemes. Then there is a morphism $\varphi: Z \rightarrow X$ such that $\pi \circ \varphi : Z\rightarrow Y$ is surjective and quasi-finite.
\end{lemma}

\begin{proof} By Noetherian induction on $Y$, it is enough to find a map $\varphi :Z \rightarrow X$ such that the map $\pi \circ \varphi$ is quasi-finite and its image contains a non-empty open set. We can assume that $Y$ is affine and irreducible. Since the claim depends only on the underlying topological space, we can assume that $Y$ is reduced. Let $\eta$ be the generic point of $Y$ and let
{$X_\eta = \pi^{-1}(\eta)$.}
 By the {assumption that $\pi$ is surjective}, $X_\eta$ is non-empty. Since $\pi$ is locally of finite type, there are affine open sets $\Spec A \subset Y$, {and} $\Spec B \subset X \times_Y \Spec A$ such that $B$ is a finitely generated $A$-algebra, $\eta \in \Spec A$, and $\Spec(B \otimes_A k(\eta))\neq \emptyset$.

It follows that there is a finite extension $L$ of $k(\eta)$ and a non-trivial map ${\nu}:B \otimes_A k(\eta) \rightarrow L$. Fix generators $b_1,\dots,b_n$ of $B$ over $A$. Let $a\in A$ be the product of the denominators of the coefficients of the minimal polynomials of $\nu(b_i\otimes 1)$ over $k(\eta)=\Frac(A)$. We obtain that $\nu(B{\otimes 1})[a ^{-1}]$
is an integral extension of $A[a ^{-1}]$. Taking $Z=\Spec \nu(B{\otimes 1})[a ^{-1}]$, we get a map $\varphi:Z \rightarrow X$ that $\pi \circ \varphi :Z \rightarrow \Spec A[a ^{-1}]$ is finite (and hence quasi-finite) and surjective. \end{proof}

{The proof of  Theorem \ref{lem:key} is based on subsequent analysis of its  special cases:}
\begin{lemma} \label{lem:key.scheme} Theorem \ref{lem:key} holds if $\frak{X}$ is a scheme.
\end{lemma}

\begin{proof} By Lemma \ref{lem:q.f.section}, there is a morphism $\varphi :Z \rightarrow X$ such that $\zeta:=\pi \circ \varphi$ is surjective and quasi-finite. {It is well known that, for quasi-finite maps, there exists} $m$ such that {$[k(z):k(\zeta(z))]<m$}, for every schematic point $z\in Z$.
It is easy to see that $\pi$ is $(m,\mathcal{F})$-almost onto.
\end{proof}

\begin{corollary} \label{cor:one.for.all} Let $\mathfrak{X}$ be a {(Noetherian)} stack and let $\mathcal{S} \subset \mathcal{F}$. Suppose that $s:Y \rightarrow \mathfrak{X}$ is a surjective $(n,\mathcal{S})$-almost onto map, and let $\pi :X \rightarrow \mathfrak{X}$ be surjective {map}. Then there is $N$ such that $\pi$ is $(N,\mathcal{S})$-almost onto.
\end{corollary}

\begin{proof} Consider the diagram
\[
\xymatrix{X \times_\frak{X} Y \ar@{->}[r]^{\pi^*} \ar@{->}[d]^{s*} & Y \ar@{->}[d]^{s} \\ X \ar@{->}[r]^{\pi} & \frak{X}}
\]

{Since $\pi$ is surjective, b}y definition, the map $\pi ^*$ is surjective. {Let $p:Z\to X \times_\frak{X} Y$ be an etale cover of the algebraic space $X \times_\frak{X} Y$ by a scheme $Z$. We get that $\pi^* \circ p$ is surjective} so, by Lemma \ref{lem:key.scheme}, {$\pi^* \circ p$} is $(m,\mathcal{F})$-almost onto, for some $m$ {and, therefore, so is $\pi^*$}. It follows that the composition $s\circ \pi ^* = \pi \circ s^*$ is $(nm,\mathcal{S})$-almost onto, which implies that $\pi$ is $(nm,\mathcal{S})$-almost onto.
\end{proof}

\begin{lemma} \label{lem:etale.almost.onto} Theorem \ref{lem:key} holds if $\mathfrak{X}$ is an algebraic space and $\pi :X \rightarrow \mathfrak{X}$ is etale.
\end{lemma}

\begin{proof} \cite[Proposition 1.2]{LMB}
states that the following are true:\begin{enumerate}
\item\label{it:eq1} The two projections $p_{1,2}:X \times_\mathfrak{X} X \rightarrow X$ are etale and the subscheme $X \times_{\mathfrak{X}} X \subset X \times X$ is an equivalence relation.
\item\label{it:eq2} $\mathfrak{X}$ is the coequalizer of
 $X \times_\mathfrak{X} X {\underset{p_2}{\overset{p_1}{\rightrightarrows}}} X$
  (as sheaves on the big etale site of $S$). In particular, for any field $F$, $\mathfrak{X}(F^{sep})=X(F^{sep})/X \times_\mathfrak{X} X(F^{sep})$.
\end{enumerate}

Note also that, for any Galois extension $F \subset L$, there is an injective map $i_{F,L}:\mathfrak{X}(F) \rightarrow \mathfrak{X}(L)$ and the image is the collection of $\Gal(L/F)$-invariants.

Statement {\eqref{it:eq1}} and the assumption that $X$ is Noetherian imply that there is $N$ such that, for any field $F$, the sizes of the equivalence classes of $X \times_\mathfrak{X} X(F^{sep})$ are at most $N$. We will show that $\pi$ is $(N,\mathcal{F})$-almost onto.

Suppose that $F$ is a field and $x\in \mathfrak{X}(F)$. By {\eqref{it:eq2}}, there is $y\in X(F^{sep})$ such that $i_{F,F^{sep}}(x)$ is the equivalence class $[y]$ of $y$. Since $\Gal_F$ preserves $[y]$, it follows that there is a closed subgroup $H \subset \Gal_F$ of index at most $N$ such that $H$ fixes $y$. Let $L=(F^{sep})^H$. Then $[L:F] \leq N$ and $y\in X(L)$. Finally, since $i_{L,F^{sep}}(\pi(y))=[y]=i_{F,F^{sep}}(x)=i_{L,F^{sep}}(i_{F,L}(x))$, it follows that $\pi(y)=i_{F,L}(x)$.
\end{proof}

Lemma \ref{lem:etale.almost.onto} and Corollary \ref{cor:one.for.all} give the following:

\begin{corollary} \label{cor:alg.space} Theorem \ref{lem:key} holds if $\frak{X}$ is an algebraic space.
\end{corollary}

\begin{lemma} \label{lem:Galois} Let $F$ be a field, let $i:C \rightarrow X$, $f:X \rightarrow Y$ be morphisms of $F$-schemes, and denote the structure map of $C$ by $\kappa: C \rightarrow \Spec F$. Suppose that $f$ is surjective and that, for some finite extension $F \subset L$, there is $\eta :\Spec L \rightarrow {Y}$ making the diagram
{
$$\xymatrix{C_L \ar@{->}[r]^{i_L} \ar@{->}[d]^{\kappa_L} & X \ar@{->}[d]^{f} \\ \Spec L \ar@{->}[r]^{\eta} & Y}$$
}
a pullback diagram (here, $i_L$ and $\kappa_L$ are the base-changes of $i$ and $\kappa$ respectively). Then there is $\zeta :\Spec F \rightarrow Y$ such that $\zeta_L=\eta$ and such that the diagram
\[
\xymatrix{C \ar@{->}[r]^{i} \ar@{->}[d]^{\kappa} & X \ar@{->}[d]^{f} \\ \Spec F \ar@{->}[r]^{\zeta} & Y}
\]
is a pullback diagram.

{In other words, if the fiber under a surjective morphism $f:X \to Y$  of an $L$ point $\eta$ is defined over $F$ then it is in fact a fiber of an $F$-point.}
\end{lemma}

\begin{proof} Let $\alpha: \Spec L \rightarrow \Spec F$ be the map corresponding to the inclusion $F \subset L$. Taking the base change of $\alpha$ and $\kappa$, we get $\alpha_C:C_L \rightarrow C$. Define $\alpha_X:X_L \rightarrow X$ and $\alpha_Y:Y_L \rightarrow Y$ similarly.

Let $R=L \otimes_F L$. We have two maps $a,b:\Spec R \rightarrow \Spec L$ such that $\alpha \circ a=\alpha \circ b$. Let $\kappa_R:C_R \rightarrow \Spec R$ be the base change of $\kappa$.{ Consider the diagram
$$\xymatrix{
		\spec(R)  \ar@/_2.0pc/@{->}[rr]_{\eta \circ a}\ar@/_3.5pc/@{->}[rr]_{\eta \circ b} \ar@/^/@{->}[r]^{a}\ar@/_/@{->}[r]_{b} & \spec(L)\ar@{->}[r]^{\eta} &Y}$$
We will show that $\eta \circ a=\eta \circ b$.
 By faithfully flat descent this would imply that $\eta$ factor through an $F$-point $\zeta:\spec(F)\to Y$. This will give  a morphism $C\to f^{-1}(\zeta)$ that becomes an isomorphism after extending scalars to $L$. Since $\spec L \to \spec F$ is faithfully flat this implies $C= f^{-1}(\zeta)$ as required.

Consider the following Cartesian squares
$$\xymatrix{C^a_R \ar@{->}[r]^{a_C} \ar@{->}[d]^{\kappa_R}& C_L\ar@{->}[r]^{i_L} \ar@{->}[d]^{\kappa_L}&X \ar@{->}[d]^{f}\\
		\spec(R)\ar@{->}[r]^{a} & \spec(L)\ar@{->}[r]^{\eta} &Y}.$$

We obviously have $C^a_R \cong C_R:= C\times_{\spec(F)} \spec(R)$. Applying the same argument to $b$ we get:

$$\xymatrix{C_R\ar@/^2.0pc/@{->}[rr]^{i_L \circ b_C}\ar@/^3.5pc/@{->}[rr]^{i_L \circ a_C} \ar@/^/@{->}[r]^{a_C}\ar@/_/@{->}[r]_{b_C} \ar@{->}[d]^{\kappa_R}& C_L\ar@{->}[r]^{i_L} \ar@{->}[d]^{\kappa_L}&X \ar@{->}[d]^{f}\\
		\spec(R)  \ar@/_2.0pc/@{->}[rr]_{\eta \circ a}\ar@/_3.5pc/@{->}[rr]_{\eta \circ b} \ar@/^/@{->}[r]^{a}\ar@/_/@{->}[r]_{b} & \spec(L)\ar@{->}[r]^{\eta} &Y}.$$

Since $i_L$ factor through $i:C\to X$, we obtain that the upper two arrows coinside (namely $i_L \circ a_C=i_L \circ b_C$). The surjectivity of $f$ implies that {$\kappa_R$} is surjective. Thus we obtain that the lower two arrows coincide (namely $\eta \circ a=\eta  \circ b$), as required. }


 \end{proof}

\begin{lemma} \label{lem:key.BG.special}
{Let $G$ be a flat group algebraic space over a scheme $X$, let $\frak{X}=[X/G]$ be the  classifying space  of $G$ (see \cite[89.13]{SP}), and let $\pi :X \rightarrow \frak{X}$ be the  neutralizing map.  Then Theorem \ref{lem:key} holds for $\pi :X \rightarrow \frak{X}$.}
\end{lemma}

\begin{proof}
{By Proposition \ref{prop:group.algebraic.space} {and Lemma \ref{lem:strat}}
we can assume that $G$ is a group scheme.}
Let $F$ be a field. By \cite[\href{https://stacks.math.columbia.edu/tag/0CQJ}{Lemma 89.15.4}]{SP},
a point $u:\Spec F \rightarrow [X/G]$ is a pair $(x,P)$, where $x\in X(F)$ and $P$ is a $G_x$-torsor in the fppf topology, i.e., a $G_x$-space that becomes trivial after base-change to the algebraic closure of $F$. The point $u$ factors through $X$ iff $P$ is a trivial $G_x$-torsor. Therefore, it is enough to show that there is a constant $N$ such that, for any finite field $F$ (in case $\mathfrak{X}$ is QCA $F$ can be taken to be arbitrary), any $x\in X(F)$ and any $G_x$-torsor $P$, there is a field extension $E \supset F$ of degree at most $N$ such that $P \times_{\Rami{\Spec F}} \Spec E$ is trivial. We show this holds for the two cases of the theorem.
\begin{enumerate}
\item \Rami{Finite field case:}\\ Let $n:=\max \#\pi_0(G_s)$ where $s$ ranges over all geometric points of $X$ (this maximum exists because $s \mapsto \# \pi_0(G_s)$ is constructible, see \cite[Proposition 9.7.8]{EGA4}). We will show that, for any finite field $F=\mathbb{F}_q$ and any algebraic group $H$ over $F$ with at most $n$ connected components, any $H$-torsor has a trivialization over $\mathbb{F}_{q^{(n!)^2}}$. Since finite fields are perfect, any fppf torsor is an etale one. Hence, we need to show that the map $H^1(\mathbb{F}_q,H) \rightarrow H^1(\mathbb{F}_{q^{(n!)^2}},H)$ is trivial. By Lang's theorem, it is enough to show that $H^1(\mathbb{F}_q,\pi_0H) \rightarrow H^1(\mathbb{F}_{q^{(n!)^2}},\pi_0H)$ is trivial. This map is the composition of $H^1(\mathbb{F}_q,\pi_0H) \rightarrow H^1(\mathbb{F}_{q^{n!}},\pi_0H)$ and $H^1(\mathbb{F}_{q^{n!}},\pi_0H) \rightarrow H^1(\mathbb{F}_{q^{(n!)^2}},\pi_0H)$, so it is enough to show that the second map is trivial. Note that the action of $\Gal(\overline{\mathbb{F}_q} / \mathbb{F}_{q^{n!}})$ on $\pi_0H$ is trivial, so any 1-cocycle is a homomorphism $\Gal(\overline{\mathbb{F}_q} / \mathbb{F}_{q^{n!}}) \rightarrow \pi_0H$, but any such becomes trivial when restricted to $\Gal(\overline{\mathbb{F}_q} / \mathbb{F}_{q^{(n!)^2}})$.

\item {QCA stack case:}\\
 Assume now that $\mathfrak{X}=[X/G]$ is QCA. By Proposition \ref{prop:strat.qca.1}, there is a stratification $X=\cup X_i$ such that $G|_{X_i^{red}}$ can be embedded as a closed subgroup in $\GL_n \times X_i^{red}$, for some $n$.
 Hence, {by Lemma \ref{lem:strat},} we can assume that $X$ is reduced and $G$ is \Rami{a} closed subgroup of $\GL_n \times X$. Similarly, using Proposition \ref{prop:strat.qca}, we can assume that the quotient $Z:=\GL_n \times X / G$ exists. Since the quotient map $\pi:\GL_n \times X \rightarrow Z$ is onto, Lemma \ref{lem:key.scheme} implies that there is a natural number ${N}$ such that, for every field $F$ and every $p\in Z(F)$, there is an extension $E \supset F$ of degree at most ${N}$ such that {the composition $\Spec E \rightarrow \Spec F \overset{p}{\rightarrow} Z$} factors through $\GL_n \times X$. We will show that, for any field $F$, and $x\in X(F)$ and every $G_x$-torsor $P$ defined over $F$, there is a field extension $E \supset F$ of degree at most ${N}$ such that $P \times \Spec E$ is trivial.

Let $F$ be a field, let $x\in X(F)$, and let $P$ be a $G_x$-torsor defined over $F$. The quotient $P \times \GL_n / G_x$ (where $G_x$ acts diagonally on $P \times \GL_n$) is a $\GL_n$-torsor over $F$. By Hilbert 90, this torsor is trivial. This means that there is a $\GL_n$-equivariant isomorphism $P \times \GL_n / G_x \rightarrow \GL_n$. Composing this isomorphism with the map $P \rightarrow P \times \GL_n / G_x$ that sends $p$ to $(p,1)G_x$, we get a morphism $i:P \rightarrow \GL_n$ which is $G_x$-equivariant. Since $P$ is a torsor, for some finite extension $L \supset F$, the base change $P \times \Spec L$ is trivial, so {by Lemma \ref{lem:base.change.quotient} we have} $i(P \times \Spec L)=\pi ^{-1} (w)$, for some $w\in Z(L)$. Applying Lemma \ref{lem:Galois}, $i(P)=\pi ^{-1} (z)$, for some $z\in Z(F)$. By the definition of ${N}$, there is a field extension $E \supset F$ of degree at most ${N}$ and a point $g\in i(P)(E)$. It follows that $P$ is trivial over $E$, which is what we wanted to prove.
\end{enumerate}

\end{proof}
Lemma \ref{lem:key.BG.special} and Corollary \ref{cor:one.for.all} give the following:

\begin{corollary} \label{lem:key.BG} Theorem \ref{lem:key} holds if $\frak{X}=[X/G]$, where $X$ is a scheme and $G$ is a flat group scheme over $X$.
\end{corollary}

In the following, by a gerbe, we mean a gerbe in the fppf topology, see \cite[Definition 95.27.1]{SP}.

\begin{lemma} \label{lem:key.gerbe.general} Theorem \ref{lem:key} holds if $\frak{X}$ is {a} gerbe over an algebraic space $[\frak{X}]$.
\end{lemma}

\begin{proof}
{Let $\cS$ be $\cF$ if $\fX$ is QCA and $\cF_f$ otherwise.}
Let $\tau:\frak{X} \rightarrow [\frak{X}]$ be the structure map. Consider the following diagram:
\[
\xymatrix{X \times_{[\frak{X}]} X \ar@{->}[r]^{\rho} \ar@{->}[d] & X \ar@{->}[d]^{\pi} \\ \frak{X} \times_{[\frak{X}]} X \ar@{->}[r] \ar@{->}[d] & \frak{X} \ar@{->}[d]^{\tau} \\ X \ar@{->}[r]^{\tau \circ \pi} & [\frak{X}]}
\]
The following hold: \begin{enumerate}
\item The map $\tau \circ \pi : X \rightarrow [\frak{X}]$ is surjective. By Corollary \ref{cor:alg.space}, it is $(N_1,\mathcal{F})$-almost onto, for some $N_1$.
\item $\tau:\frak{X} \rightarrow [\frak{X}]$ is a gerbe, and so $\frak{X} \times_{[\frak{X}]} X \rightarrow X$ is a gerbe by \cite[{\href{https://stacks.math.columbia.edu/tag/06QB}{Lemma 95.27.3}}]{SP}. The map $(\pi,id):X \rightarrow \frak{X} \times_{[\frak{X}]} X$ is a section, so $\frak{X} \times_{[\frak{X}]} X$ is isomorphic to {the classifying stack} $[G/X]$, for some flat group algebraic space $G$ over $X$, by \cite[{\href{https://stacks.math.columbia.edu/tag/06QB}{Lemma 95.27.6}}]{SP}. If $\mathfrak{X}$ is QCA, then so is $\frak{X} \times_{[\frak{X}]} X$.
\item \label{cond:a.o.gerbe} The map $X \times_{[\frak{X}]} X \rightarrow \frak{X} \times_{[\frak{X}]} X$ is surjective. By the previous claim and Lemma \ref{lem:key.BG}, it is {$(N_2,\cS)$-almost onto, for some $N_2$.}
\end{enumerate}
{We will prove that the map $X \times_{[\frak{X}]} X\to \fX$ is $(N_1N_2,\cS)$-almost onto. This is enough by Corollary  \ref{cor:one.for.all}}
Let $F$ be a field and let $pt:\Spec F \rightarrow \frak{X}$ be an $F$-point. There is a field extension $K \supset F$ of degree at most $N_1$ such that the composition $\Spec {K} \rightarrow \Spec {F} \rightarrow \frak{X}\rightarrow [\frak{X}]$ factors through a map $q:\Spec {K} \rightarrow X$. The map $(pt,q)$ defines a map $pt':\Spec {K}\rightarrow \frak{X} \times_{[\frak{X}]} X$. By \ref{cond:a.o.gerbe}, there is a field extension ${E} \supset {K}$ of degree at most $N_2$ such that the composition $\Spec {E} \rightarrow \Spec {K} \rightarrow \frak{X} \times_{[\frak{X}]} X$ factors through a map $r:\Spec {E} \rightarrow X \times_{[\frak{X}]} X$. It follows that the composition $\Spec {E} \rightarrow \Spec {F}\rightarrow \frak{X}$ factors through $\rho \circ r: \Spec {E} \rightarrow X$.

\end{proof}

\begin{proof}[Proof of Theorem \ref{lem:key}] {By}  \cite[{\href{https://stacks.math.columbia.edu/tag/06RB}{95.28}}]{SP},
there is a stratification of $\frak{X}$ by locally closed substacks $\frak{X}_i$ such that $\frak{X}_i$ are fppf gerbes over some algebraic spaces. Since $\mathfrak{X}$ is Noetherian, there are only finitely many $\mathfrak{X}_i$. {The assertion now follows from Lemmas \ref{lem:key.gerbe.general} and \ref{lem:strat}.}

\end{proof}

\section{Onto presentations of stacks (Proof of Theorem \ref{thm:lift})} \label{sec:lift}

The proof of Theorem \ref{thm:lift} is based on Theorem \ref{lem:key} and the following proposition:

\begin{proposition}\label{lem:alm.onto.imp.onto}
Let   $\mathcal{S} \subset \mathcal{F}_{\mathrm{perf}}$ and {let} $\fX$ be a stack.
\begin{enumerate}
 \item \label{lem:alm.onto.imp.onto:1}
If there is an $(n,\mathcal{S})$-almost onto presentation of $\fX$ by a scheme, then there is an $\mathcal{S}$-onto presentation of $\fX$ by an algebraic space.
 \item \label{lem:alm.onto.imp.onto:2} If  $\fX$ is an algebraic space and
 there is an $(n,\mathcal{S})$-almost onto presentation of $X$ by a scheme, then there is an $\mathcal{S}$-onto presentation by a scheme.
\end{enumerate}
\end{proposition}

The proof of Proposition \ref{lem:alm.onto.imp.onto} will be given in \S\ref{sec:prf.alm.onto}; the proof uses several auxiliary results which we prove in \S\ref{sec:prf.alm.onto} and \S\ref{sec:improving.pres}. We now show how to deduce Theorem \ref{thm:lift} from Proposition \ref{lem:alm.onto.imp.onto}.

We will need the following:

\begin{lemma} \label{lem:left.lift} Let $A$ be a local ring, $I$ an ideal in $A$ such that $(A,I)$ is a Henselian pair (see \cite[15.11]{SP}). Then the embedding $\Spec(A/I) \rightarrow \Spec(A)$ has the left lifting property with respect to smooth maps of schemes, i.e., for any commutative diagram
\[
\xymatrix{\Spec(A/I) \ar@{->}[r] \ar@{->}[d] & X \ar@{->}[d] \\ \Spec(A) \ar@{->}[r] & Y}
\]
such that $X,Y$ are schemes and the map $X \rightarrow Y$ is smooth, there is a map $\Spec(A) \rightarrow X$ such that the diagram
\[
\xymatrix{\Spec(A/I) \ar@{->}[r] \ar@{->}[d] & X \ar@{->}[d] \\ \Spec(A) \ar@{->}[r] \ar@{->}[ru] & Y}
\]
is commutative.
\end{lemma}

\begin{proof} Denote the map $X \rightarrow Y$ by $\phi$. There is a Zariski open cover $X=\bigcup U_i$ such that $\phi |_{U_i}$ factors as the composition of an etale map $\psi_i:U_i \rightarrow Y \times \mathbb{A} ^n$ and the projection $Y \times \mathbb{A} ^n \rightarrow Y$. Since $A/I$ is local, we can replace $X$ by some $U_i$, so it is enough to prove the claim in the following cases: \begin{enumerate}
\item $\phi$ is the projection $Y \times \mathbb{A} ^n \rightarrow Y$. The claim follows since the map $A \rightarrow A/I$ is onto.
\item $\phi$ is etale. The claim follows from the definition of Henselian pair.
\end{enumerate}
\end{proof}

\begin{corollary} \label{cor:Henselian.pair} Let $A$ be a local ring and let $I$ be an ideal in $A$ such that $(A,I)$ is a Henselian pair. Let $\phi:X \rightarrow \mathfrak{X}$ be a $\left\{ \Spec(A/I) \right\}$-onto presentation. Assume that, for any algebraic space $\mathfrak{B}$ and any $A/I$-point $r:\Spec(A/I) \rightarrow \mathfrak{B}$, there is a presentation $\psi: B \rightarrow \mathfrak{B}$ such that $r$ is $\psi$-liftable. Then $\phi$ is $\left\{ \Spec(A) \right\}$-onto.
\end{corollary}

\begin{proof} Suppose that $q:\Spec(A) \rightarrow \mathfrak{X}$ be an $A$-point. Since $\phi$ is $\left\{ \Spec(A/I) \right\}$-onto, we can lift the composition $\Spec(A/I) \rightarrow \Spec(A) \overset{q}{\rightarrow} \mathfrak{X}$ to a map $\Spec(A/I) \rightarrow X$. This gives a map $r:\Spec(A/I) \rightarrow X \times_\mathfrak{X} \Spec(A)$. By assumption, there is a scheme $B$ and a presentation $\psi :B \rightarrow X \times_\mathfrak{X} \Spec(A)$ such that $r$ is $\psi$-liftable. Let $r':\Spec(A/I) \rightarrow B$ be a lift of $r$. Applying Lemma \ref{lem:left.lift} to the diagram
\[
\xymatrix{\Spec(A/I) \ar@{->}[r]^{r'} \ar@{->}[d] & B \ar@{->}[d] \\ \Spec(A) \ar@{->}[r] & \Spec(A)}
\]
we get a map $s:\Spec(A) \rightarrow B$. The composition of $s$ and the projection to $X$ is a lift of $q$.
\end{proof}

We can now prove Theorem \ref{thm:lift}:

\begin{proof}[Proof of Theorem \ref{thm:lift}] We first show the claim replacing $\mathcal{H}_{f},\mathcal{H}_{r},\mathcal{H}_{perf}$ by $\mathcal{F}_{f},\mathcal{F}_{r},\mathcal{F}_{perf}$. Let $\phi: X \to \fX$ be a presentation of $\fX$ by a scheme $X$. By Theorem \ref{lem:key}, there exists an integer $n$ such that $\phi$ is $(n,\cF_f)$-almost onto {(or $(n,\cF_{})$}-almost onto if $\fX$ is QCA). By definition, $\phi$ is also $(2,\cF_r)$-almost onto.  By Proposition \ref{lem:alm.onto.imp.onto}\eqref{lem:alm.onto.imp.onto:1}, there exists an algebraic space $X'$ and a presentation $\psi : X' \rightarrow \mathfrak{X}$ which is $(\cF_r\cup \cF_f)$-onto ($\cF_{\mathrm{perf}}$-onto if $\mathfrak{X}$ is QCA). Let $X'' \rightarrow X'$ be a presentation of $X'$ by a scheme $X''$.  Since $X'$ is $QCA$, applying Theorem \ref{lem:key} and Proposition \ref{lem:alm.onto.imp.onto}\eqref{lem:alm.onto.imp.onto:2}, we obtain a scheme $X'''$ and an $\cF_{\mathrm{perf}}$-onto presentation $\psi :X''' \to X'$. The composition $\phi \circ \psi:X''' \rightarrow \mathfrak{X}$ is a presentation which is $(\cF_r\cup \cF_f)$-onto ($\cF_{\mathrm{perf}}$-onto if $\mathfrak{X}$ is QCA).

The claim of the theorem for $\mathcal{H}_{f},\mathcal{H}_{r},\mathcal{H}_{perf}$ follows now by Corollary \ref{cor:Henselian.pair}.
\end{proof}

\subsection{Internal Hom} \label{sec:internal.hom}

\begin{definition}
Let $X, Y$ be {$S$-}schemes. Let $X^{\wedge}_S Y$ be the contravariant functor from the category of $S$-schemes to the category of sets defined by $$(X^{\wedge}_S Y)(T):=Mor(Y\times_S T,X).$$ \Rami{If the base scheme $S$ is clear from the context we will omit it from the notation or simply denote this functor by $X^{Y}$}
\end{definition}

\begin{lemma} \label{lem:power.representable}
$ $
\begin{enumerate}
\item If $Y \rightarrow S$ is finite and etale,  then $X^{\wedge}_S Y$ is representable by an algebraic space (of finite type).
\item If $Y \rightarrow S$ is finite and etale,  and $X$ is quasi-projective, then $X^{\wedge}_S Y$ is representable by a scheme (of finite type).
\end{enumerate}
\end{lemma}
Although the statement is standard, we did not find a complete proof in the literature, so we deduce it from a similar statement appearing in \cite[I \S1 6.6]{DeGa}. For another version, see \cite{Ol}. For the proof we will need the following simple lemmas:
\begin{lemma} \label{lem:etale.trivial} Let $S$ be a connected scheme and let $S'\to S$ be a finite etale map. Then there exists an etale cover $\eta: T\to S$ such that $$T\times _S S' \cong T \sqcup \cdots  \sqcup T,$$
as $T$-schemes.
\end{lemma}
\begin{proof}
The proof is by induction on the degree $d$ of the map $S'\to S$.
Without loss of generality, we may assume that $S'$ is connected.
 the base $d=0$ is obvious.  Without loss of generality we can assume that $S' \to S$ is a cover. Consider the diagram
$$\xymatrix{\Delta S' \ar@{->}[r]& S'\times_S S' \ar@{->}[r]^{}  \ar@{->}[d]^{} & S' \ar@{->}[d]^{} \\& S' \ar@{->}[r]^{}& S}$$

Let $U:=S'\times_S S' \smallsetminus \Delta S'$. The map $U \to S'$ is finite etale map of degree $d-1$. Thus by induction assumption there is an etale cover $T \to S'$ such that
$$T\times _S' U \cong T \sqcup \cdots  \sqcup T.$$ Composing $T \to S' \to S$ we obtain the required cover.
\end{proof}

\begin{lemma}[{See e.g. \cite[Lemma 34.20.10]{SP}}] \label{lem:finite.type.local}
The property of being of finite type is local in the fppf topology on the target. Namely, let $X\to S$  be a (not necessarily finite type) $S$-scheme and $\phi:T\to S$ be a faithfully flat morphism of finite type. Assume that $X\times_S T$ is {of} finite type over $T$. Then, $X\to S$  is of finite type.
\end{lemma}

\begin{proof}[Proof of Lemma \ref{lem:power.representable}]$ $
\begin{enumerate}
\item By faithfully flat descent, $X^{\wedge} Y$ is a sheaf in the fpqc topology and in particular in the etale topology. Thus, we need to find a scheme $A$ together with an etale cover $A \to X^{\wedge} Y$. Without loss of generality we can assume that $S$ is connected. By Lemma \ref{lem:etale.trivial} there exists an etale cover $T \to S$ such that $T\times _S Y \cong \underbrace{T \sqcup \cdots  \sqcup T}_{\text{$n$ copies}}$ as $T$-schemes. Since $T \rightarrow S$ is an etale cover, the map $X^{\wedge}_S Y \times_S T \rightarrow X^{\wedge}_S Y$ is an etale cover. Since
\[
X^{\wedge}_S Y \times_S T=(X\times_S T)^{\wedge}_T (Y \times_S T) = (X \times_S T)^{\wedge}_T(T \sqcup \cdots  \sqcup T),
\]
we have that $X^{\wedge}_S Y \times_S T$ is representable by $(X \times_S T)^n$.
\item By {\cite[I \S1 6.6]{DeGa}}  $X^{\wedge}_S Y$ is representable by a scheme which is not a-priori of finite type. Let $T \rightarrow S$ be the etale cover from the previous part and consider the Cartesian square
$$\xymatrix{ (X \times_S T)^n \ar@{->}[r]^{}  \ar@{->}[d]^{} & X^{\wedge}_S Y \ar@{->}[d]^{} \\ T \ar@{->}[r]^{}& S}$$
The horizontal arrows are etale covers, and the morphism $ (X \times_S T)^n \to T$  is of finite type. Therefore, by Lemma \ref{lem:finite.type.local},  so is the morphism   $X^{\wedge}_S Y\to S$.

\end{enumerate}
\end{proof}
\begin{corollary} \label{cor:hom.diagram.representable}
Let $\mathcal{C}$ be a finite category and let $D_1,D_2:\mathcal{C} \to RamiA{Sch_{/S}}$ be two functors. Let ${D_1}^{\wedge}_S D_2$ be the contravariant functor from the category of $S$-schemes to the category of sets defined by $$({D_1}^{\wedge}_S D_2)(T):=Mor(D_2\times T,D_1),$$ where $D_2\times T$  is the composition of product with $T$ and $D_2$.
Assume that the image of $D_2$ consists of schemes which are finite and etale over $S$. Then
\begin{enumerate}
\item\label{cor:hom.diagram.representable:1} ${D_1}^{\wedge}_S D_2$ is representable by an algebraic space.
\item\label{cor:hom.diagram.representable:2} If the image of $D_1$ consist of quasi-projective schemes then   ${D_1}^{\wedge}_S D_2$ is representable by a scheme.
\end{enumerate}
\end{corollary}

\begin{proof}
\Rami{We first prove \eqref{cor:hom.diagram.representable:1}.
${D_1}^{\wedge}_S D_2$ is a limit of a finite diagram of functors, each represented by an algebraic space. By the Yoneda Lemma, any morphism between such functors comes from a  morphism of algebraic spases. The asertion follows now from the fact that the category of algebraic spaces is closed under {finite} limits.
Part  \eqref{cor:hom.diagram.representable:2} is proved in a similar way.}
\end{proof}

In the rest of the section, we will not distinguish between representable functors and their representing objects.

\subsection{Improving a presentation} \label{sec:improving.pres}

\begin{notation}
Let $\phi:X \to \fX$ be a presentation of an algebraic stack defined over a scheme $S$.  Denote by $[\phi]_\bullet$ the simplicial scheme given by $[\phi]_1:=X$,  $[\phi]_n:=X \times_\fX X_{n-1}$ with the standard boundary and degeneration maps. Denote by $[\phi]_{\bullet\leq 3}$ be the full subdiagram of $[\phi]_\bullet$  with {vertices} $[\phi]_1,[\phi]_2,[\phi]_3$.
\end{notation}
Note that for two maps $\phi$ and $\phi'$ as above we have a canonical isomorphism $[\phi]_\bullet ^{[\phi']\bullet}\cong [\phi]_{\bullet\leq 3} ^{[\phi']_{\bullet\leq 3}}$.

The goal of this subsection is to prove the following:

\begin{lemma} \label{lem:improving.pres} Let $\mathfrak{X}$ be an algebraic stack defined over a scheme $S$, let $X$ be a scheme over $S$, let $\phi :X \rightarrow \mathfrak{X}$ be a presentation and let $\psi :S' \rightarrow S$ be a finite etale {and onto} map. Then the functor  $[\phi]_\bullet ^{[\psi]\bullet}$  is representeble by algebraic space and there is a presentation $f_{\phi,\psi}:[\phi]_\bullet ^{[\psi]\bullet} \rightarrow \mathfrak{X}$ such that, if $T$ is an $S$-scheme and $x:T \to \mathfrak{X}$ is such that the natural map $T\times_S S' \to \mathfrak{X}$ is  $\phi$-liftable, then $x$ is  $f_{\phi,\psi}$-liftable.

Moreover, if $\mathfrak{X}$ is an algebraic space and $X$ quasi-affine, then $[\phi]_\bullet ^{[\psi]\bullet}$ is a scheme.
\end{lemma}
In order to build $f_{\phi,\psi}$  we need to discuss the notion of descent data.

\begin{definition}
Let $\fX$ be an algebraic stack defined over a scheme $S$. Let $\psi:Y'\to Y$ be an etale map of $S$-schemes.
A {descent datum for $\fX$ with respect to $\psi$ (or a $\psi$-descent datum for $\mathfrak{X}$)} is a map $s:Y' \to \fX$ and an isomorphism $F$ between $s\circ d_1$ and $s\circ d_2$ (where $d_i:[\psi]_2\to [\psi]_1$ are the boundary maps) satisfying the cocycle condition $(F\circ d_{12})(F \circ d_{23})=F\circ d_{13}$
(see \cite[8.3, 8.4]{SP}).
\end{definition}
The collection of all $\psi$-descent data forms a groupoid. We have a natural functor from $\fX(Y)$ to this groupoid. $\mathfrak{X}$ being a stack implies that this functor is an equivalence.

\begin{notation}
Let $ \fX$ be an algebraic stack defined over a scheme $S$. Let $\psi:S'\to S$ be a finite etale map.
Define a {functor} $\fX_{S'/S}$ {from $S$-schemes to groupoids} by  $$\fX_{S'/S}(Y)=\{\text{ descent data for $\fX$  with respect to $Y\times_S S' \to Y$ } \}.$$

\end{notation}
By the discussion above, there is a natural equivalence of functors $\fX \to \fX_{S'/S}$. In particular, $\mathfrak{X}_{S'/S}$ is a stack naturally identified with $\fX$, for any such $S'$.

\begin{definition} Let $\phi:X \to \mathfrak{X}$ be a presentation of an algebraic stack defined over a scheme $S$. Let $\eta:Y'\to Y$ be an etale map of $S$-schemes. An explicit descent datum for a map from $Y$ to $\mathfrak{X}$ with respect to $\phi$ and $\eta$ is a
morphism of diagrams $[\eta]_{\bullet \leq 3}\to [\phi]_{\bullet \leq 3}$.
\end{definition}
Any {explicit} {descent} datum for a map from $Y$ to $\mathfrak{X}$ with respect to $\phi$ and $\eta$  gives a {descent} datum for  a map from $Y$ to $\mathfrak{X}$ with respect to $\phi$.
We are now ready to define $f_{\phi,\psi}$.
\begin{definition}
Let $\phi:X \to \mathfrak{X}$ be a presentation of an algebraic stack defined over a scheme $S$ and let $\psi :S' \rightarrow S$ be a finite etale {and onto} map. In view of the discussion above we obtain a natural map $[\phi]_{\bullet \leq 3}^{[\psi]_{\bullet \leq 3}} \to \mathfrak X_{S'/S}$. 
This gives us the map $f_{\phi,\psi} :[\phi]_{\bullet}^{[\psi]_{\bullet}} \to \mathfrak X$
\end{definition}


In order to prove that $f_{\phi,\psi}$ is a presentation, we will use the following

\begin{lemma} \label{lem:base.change.presentation}
Let $\phi:X \to Y$ be a morphism of $S$-schemes and $T$ be an $S$-scheme.
\begin{enumerate}
\item If $\phi \times_S T :X\times_S T  \to Y\times_S T$  and $T \to S$  are surjective morphisms, then so is $\phi$
\item If $\phi \times_S T :X\times_S T  \to Y\times_S T$  is smooth and $T \to S$  is surjective and smooth, then  $\phi$ is smooth.
\item Suppose that $\mathfrak{X}$ is a stack over $S$, $\psi : X \rightarrow \mathfrak{X}$ is an $S$-morphism, and $T \rightarrow S$ is a surjective and smooth morphism. If $\psi \times_S T$ is a presentation, then so is $\psi$.
\end{enumerate}
\end{lemma}
\begin{proof}
$ $
\begin{enumerate}
\item We denote the underlying topological space of a scheme $A$ by $|A|$. By definition, a map $A\to B$ is surjective iff $|A|\to |B|$ is surjective.

Consider the commutative diagram
$$\xymatrix{X\times_S T \ar@{->}[r]^{\phi \times_S T}  \ar@{->}[d]^{pr_X} & Y\times_S T \ar@{->}[d]^{pr_Y} \\ X \ar@{->}[r]^{\phi}& Y}.$$
It gives rise to a commutative diagram

$$\xymatrix{ & |Y|\times_{|S|} |T| \ar@{->}[d]^{} \ar@/^4pc/[dd]^{pr_{|Y|}}\\
|X\times_S T| \ar@{->}[r]^{|\phi \times_S T|}  \ar@{->}[d]^{|pr_X|} & |Y\times_S T| \ar@{->}[d]^{|pr_Y|} \\ |X| \ar@{->}[r]^{|\phi|}& |Y|}$$
Since the map $T\to S $ is surjective, so is $|T|\to|S|$ and thus so is $pr_{|Y|}$. This implies that $|pr_{Y}|$ is surjective. Together with the fact that $|\phi \times_S T|$ is surjective this implies that $|\phi|\circ |pr_X| $ is surjective, which implies the assertion.

\item As before $pr_X$ is surjective. Also, since $T\to S$  is smooth so is $pr_X$. Thus by \cite[Lemma 34.11.4]{SP} it is left to show that $\phi \circ pr_X$ is smooth. This follows from the fact that $\phi \times_S T$ and $pr_Y$ are smooth.

\item This follows from the previous claims.
\end{enumerate}

\end{proof}
\begin{corollary}
Let $\phi:X \to \fX$ be a morphism of an $S$-scheme to an $S$-stack and let $T\to S$ be a surjective smooth morphism of schemes. Assume that $\phi \times_S T :X\times_S T  \to \fX\times_S T$  is a presentation. then so is $\phi$.
\end{corollary}


We will also need the following:

\begin{lemma} \label{lem:1.to.1.implies.quasi.affine}
Let $\phi:X \to Y$ be a morphism of schemes such that for any scheme $T$, the map $\phi(T):X(T) \to Y(T)$ is 1-1. Then $\phi$  is quasi-affine.
\end{lemma}
\begin{proof}
We may assume $Y$ is separated. {The assumption implies that the diagram
\[
\xymatrix{\Delta X \ar@{->}[r] \ar@{->}[d] & X \times X \ar@{->}[d] \\ \Delta Y \ar@{->}[r] & Y \times Y}
\]
is cartesian, and, therefore, $$\Delta(X)=(\phi\times\phi)^{-1}(\Delta Y).$$ This implies that $X$ is separated. The assumption also implies that $\phi$ is quasi finite. Thus by \cite[Lemma 36.38.2]{SP}, $X$ is quasi-affine.}
\end{proof}


\begin{corollary} \label{cor:quasi.affine.presentation}
Let $\phi:X \to Y$ be a presentation of an algebraic space by a quasi-affine scheme. Then $[\phi]_i$ are quasi-affine.
\end{corollary}
\begin{proof}
The two restriction maps give a morphism $[\phi]_2\to [\phi]_1\times [\phi]_1=X\times X$. This morphism satisfies the conditions  of Lemma \ref{lem:1.to.1.implies.quasi.affine} and, thus, it is quasi-affine.  This implies that $[\phi]_2$ is quasi-affine. Since, for $i>2$, we have $[\phi]_i=[\phi]_{i-1}\times_{X} [\phi]_{2}$, we obtain by induction that $[\phi]_i$  is  also quasi-affine.
\end{proof}

\begin{proof}[Proof of Lemma \ref{lem:improving.pres}]
Since $[\phi]_\bullet ^{[\psi]\bullet}\cong [\phi]_{\bullet\leq 3} ^{[\psi]_{\bullet\leq 3}}$, Corollary \ref{cor:hom.diagram.representable} implies that $[\phi]_\bullet ^{[\psi]\bullet}$ is representeble by an algebraic space.
It follows from the definitions that, if $T$ is an $S$-scheme and $x:T \to \mathfrak{X}$ is such that the natural map $T\times_S S' \to \mathfrak{X}$ is  $\phi$-liftable, then $x$ is  $f_{\phi,\psi}$-liftable.

If $\mathfrak{X}$ is an algebraic space and $X$ is quasi-affine, then $\phi_i$ are quasi-affine, and, by Corollary \ref{cor:hom.diagram.representable}, $[\phi]_\bullet ^{[\psi]\bullet}$ is a scheme of finite type.

It remains to prove that $f_{\phi,\psi}$ is a presentation.
\begin{enumerate}[{Case} 1:]
\item $S'=S \sqcup \cdots \sqcup S$ and $\psi$ is the projection.\\
In this case it is easy to see that $[\phi]_\bullet ^{[\psi]_\bullet} \cong X \times_X [\phi]_2 \times_X \cdots \times_X [\phi]_2$ where the maps $[\phi]_2 \to X$ are the first boundary maps, and the number of appearances of $ \times_Y$ and $\sqcup$ is the same.
It is also easy to  see that under this identification the map $f_{\phi,\psi}$ is  the composition of the projection to $X$ and the $\phi$. This proves the assertion.
\item The general case\\
 By Lemma \ref{lem:etale.trivial} there exists an etale cover  $\eta: T\to S$ such that $T\times _S S' \cong T \sqcup \cdots  \sqcup T$  as an $S'$  scheme.
By Lemma \ref{lem:base.change.presentation} it is  enough to show that $f_{\phi,\psi} \times_S T:[\phi]_\bullet ^{[\psi]_\bullet}\times_S T \to \fX \times_S T$ is a presentation. Equvivalently we have to show that
$$f_{\phi\times T,\psi\times T} :({[\phi \times T]_\bullet})\,^\wedge_{T}\, ([\psi\times T]_\bullet) \to \fX \times T$$ is a presentation.
This follows from the previous case.
\end{enumerate}

\end{proof}

\subsection{Proof of Proposition \ref{lem:alm.onto.imp.onto}} \label{sec:prf.alm.onto}

\begin{notation} Let $n$ be a positive integer. Let $\mathbb{U}_n \subset \mathbb{A} ^1 \sqcup \cdots \sqcup \mathbb{A} ^{n}$ be the $\mathbb{Z}$-scheme of separable monic polynomials of degree at most $n$ and let  $\mathbb{U}_n' = \left\{ (f,a)\in \mathbb{U} \times \mathbb{A}^1 \mid f(a)=0 \right\}$. Note that there is an obvious finite etale and onto map $\mathbb{U}_n' \rightarrow \mathbb{U}_n$.
\end{notation}

\begin{proof}[Proof of Proposition  \ref{lem:alm.onto.imp.onto}] Let $X \to \fX$ be an $(n,\cS)$-almost onto presentation of a stack by  a scheme. Without loss of generality, we may assume that $X$ is affine.

Let $S_n=S \times_{\Spec \mathbb{Z}} \mathbb{U}_n$, $S_n'=S\times_{\Spec \mathbb{Z}} \mathbb{U}_n'$, $\mathfrak{X}_n:=\mathfrak{X} \times_{\Spec \mathbb{Z}} \mathbb{U}_n$, and $X_n:=X \times_{\Spec \mathbb{Z}} \mathbb{U}_n$. Applying Lemma \ref{lem:improving.pres} to $(S_n,S_n',\mathfrak{X}_n,X_n)$ instead of $(S,S',\mathfrak{X},X)$, we get a presentation
$$({[\phi_n]_\bullet})\,^\wedge_{S_n}\, ([\psi_n]_\bullet)
\rightarrow \mathfrak{X}_n$$
 of $S_n$-stacks. The composition $$({[\phi_n]_\bullet})\,^\wedge_{S_n}\, ([\psi_n]_\bullet)\rightarrow \mathfrak{X}_n \rightarrow \mathfrak{X}$$ is an $\mathcal{S}$-onto presentation by an algebraic space. This proves part \ref{lem:alm.onto.imp.onto:1}. Since $\mathbb{U}_n$ is quasi-affine so is $X_n$. Thus if  $\mathfrak{X}$ is an algebraic space then $({[\phi_n]_\bullet})\,^\wedge_{S_n}\, ([\psi_n]_\bullet)$ is a scheme. This proves part \ref{lem:alm.onto.imp.onto:2}.


\end{proof}

\appendix
\section{Group schemes and their classifying spaces} \label{sec:group.schemes}
In this appendix we will {deduce some statements about group algebraic spaces from the corresponding statements for algebraic groups. The statements are about existence of some stratification of the base, so the transition between algebraic groups and group algebraic spaces is standard. We included it for completeness, since we could not find them in the literature for the generality of group schemes.}
\subsection{The generic point}
\begin{notation}For any scheme $X$ denote by $Alg(X)$ the category of algebraic spaces of finite type over $U$. We consider the assignment $X \mapsto Alg(X)$ as a contravariant 2-functor to the 2-category of categories.
\end{notation}

\begin{proposition} \label{prop:gen.pt}
{Assume that $S$ is irreducible and reduced}, and let $\eta$ be its generic point.
Then the natural functor $$\lim_{\underset{U \subset S}{\longrightarrow}}Alg(U) \to Alg(\eta)$$ is an equivalence of categories.
\end{proposition}

The affine case follows from a standard argument:

\begin{lemma} \label{lem:affine.case} For a scheme $X$, let $Aff(X)$ be the category of schemes affine over $X$. If $S$ is a reduced and irreducible scheme with generic point $\eta$, then the natural functor
\[
\lim_{\underset{U \subset S}{\longrightarrow}}Aff(U) \rightarrow Aff(\eta)
\]
is an equivalence of categories.
\end{lemma}

\begin{lemma} \label{lem:two.diagrams} Let $X,Y$ be algebraic spaces and let $\varphi_1,\varphi_2:X \rightarrow Y$ be two morphisms. Then there are affine schemes $\widetilde{X},\widetilde{Y}$, etale covers $\pi_X:\widetilde{X} \rightarrow X,\pi_Y:\widetilde{Y} \rightarrow Y$, and morphisms $\widetilde{\varphi_1},\widetilde{\varphi_2}:\widetilde{X} \rightarrow \widetilde{Y}$ such that the diagrams
\[
\xymatrix{\widetilde{X} \ar@{->}[r]^{\widetilde{\varphi_1}} \ar@{->}[d]^{\pi_X} & \widetilde{Y}  \ar@{->}[d]^{\pi_Y} \\ X \ar@{->}[r]^{\varphi_1} & Y} \quad\quad \xymatrix{\widetilde{X} \ar@{->}[r]^{\widetilde{\varphi_2}} \ar@{->}[d]^{\pi_X} & \widetilde{Y}  \ar@{->}[d]^{\pi_Y} \\ X \ar@{->}[r]^{\varphi_2} & Y}
\]
commute.
\end{lemma}

\begin{proof} Let $\pi_Y:\widetilde{Y} \rightarrow Y$ be an etale cover of $Y$ {by an affine scheme}. For $i=1,2$, let ${X}_i=\widetilde{Y} \times_Y X$, where the map $X \rightarrow Y$ is $\varphi_i$. Let ${X}_3={X_1} \times_X {X_2}$, let $\widetilde{X} \rightarrow X_3$ be an etale cover of $X_3$ {by an affine scheme}, and let $\pi_X:\widetilde{X} \rightarrow X$ be the composition $\widetilde{X} \rightarrow X_3 \rightarrow X$. It is easy to see that $\pi_X$ is etale and that there are $\widetilde{\varphi_i}$ as requested by the lemma.
\end{proof}
{The following is standard.}
\begin{lemma} \label{lem:generic.cover} If {a morphism} $X \rightarrow Y$ between schemes is an etale (respectively, Zariski) cover over the generic point, then it is an etale (respectively, Zariski) cover over an open set.
\end{lemma}

\begin{proof}[Proof of Proposition \ref{prop:gen.pt}] $ $
\begin{description}
\item[Faithful:] Let $U \subset S$ be an open set, $X,Y\in Alg(U)$, and $\varphi_1,\varphi_2:X \rightarrow Y$ be morphisms such that $\varphi_1|_\eta:X_\eta \rightarrow Y_\eta$ is equal to $\varphi_2|_\eta:X_\eta \rightarrow Y_\eta$. We need to show that there is an open $U' \subset U$ such that $\varphi_1|_{U'}=\varphi_2|_{U'}$. Apply Lemma \ref{lem:two.diagrams} to get the diagrams
\[
\xymatrix{\widetilde{X} \ar@{->}[r]^{\widetilde{\varphi_1}} \ar@{->}[d]^{\pi_X} & \widetilde{Y}  \ar@{->}[d]^{\pi_Y} \\ X \ar@{->}[r]^{\varphi_1} & Y} \quad\quad \xymatrix{\widetilde{X} \ar@{->}[r]^{\widetilde{\varphi_2}} \ar@{->}[d]^{\pi_X} & \widetilde{Y}  \ar@{->}[d]^{\pi_Y} \\ X \ar@{->}[r]^{\varphi_2} & Y}.
\]
By Lemma \ref{lem:affine.case}, there is an open set $U' \subset U$ such that $\widetilde{\varphi_1}|_{U'}=\widetilde{\varphi_2}|_{U'}$. Since $\pi_X|_{U'}$ is an etale cover, it is an epimorphism, and therefore $\varphi_1|_{U'}=\varphi_2|_{U'}$.
\item[Full:] Let $U \subset S$ be open, $X,Y\in Alg(U)$, and {$\varphi:X_{\eta} \rightarrow Y_\eta$}. We need to show that there is an open $U' \subset U$ and $\psi:X_{U} \rightarrow Y_U$ such that $\psi|_\eta=\varphi$. Let $\pi_X:\widetilde{X} \rightarrow X$ and $\pi_Y:\widetilde{Y} \rightarrow Y$ be etale covers {by affine schemes}. We have maps $\widetilde{X}_\eta \rightarrow X_\eta \rightarrow Y_\eta$. Let $p_\mathcal{Z}:\mathcal{Z} \rightarrow \widetilde{X}_\eta \times_{Y_\eta} \widetilde{Y}_\eta$ be an etale cover {by an affine scheme}. By Lemma \ref{lem:affine.case}, there is an open subset $U' \subset U$, a scheme $Z\in Aff(U')$, morphisms $\alpha:Z \rightarrow \widetilde{X}$ and $\beta:Z \rightarrow \widetilde{Y}$, and an isomorphism $Z_\eta \rightarrow \mathcal{Z}$ such that $\alpha{|}_\eta$ is equal to the composition $Z_\eta \rightarrow \mathcal{Z} \rightarrow \widetilde{X}_\eta \times_{Y_\eta} \widetilde{Y}_\eta\rightarrow \widetilde{X}_\eta$ and $\beta{|}_\eta$ is equal to the composition $Z_\eta \rightarrow \mathcal{Z} \rightarrow \widetilde{X}_\eta \times_{Y_\eta} \widetilde{Y}_\eta\rightarrow \widetilde{Y}_\eta$. By lemma \ref{lem:generic.cover}, there is an open subset $U'' \subset U'$ such that the {restriction $\alpha{|}_{U''}:Z_{U''} \rightarrow \widetilde{X}_{U''}$} is etale. Let $\gamma$ be the composition $Z \rightarrow \widetilde{Y} \rightarrow Y$, and let $p_1,p_2:Z \times_X Z \rightarrow Z$ be the two projections. Since $(\gamma \circ p_1){|}_\eta=(\gamma \circ p_2){|}_\eta$, the faithfullness implies that there is an open subset $U''' \subset U''$ such that $(\gamma \circ p_1){|}_{U'''}=(\gamma \circ p_2){|}_{U'''}$. By {faithfully flat} descent, there is a morphism $\psi:X_{U'''} \rightarrow Y_{U'''}$ such that {the composition $Z\to X\overset{\psi}{\to} Y$ is equal to $\gamma|_{U'''}$. Since $Z_\eta\to X_\eta$ is epimorphism, this implies that} $\psi{|}_\eta={\varphi}$.
\item[Essentially surjective:] We divide the proof to the following steps
\begin{enumerate}
\item {We prove that i}f $\mathcal{X}$ is a separated scheme over $\eta$, then there is an open $U \subset S$ and a scheme $X$ over $U$ such that $X_\eta=\mathcal{X}$:\\
Let $\widetilde{\mathcal{X}} \rightarrow \mathcal{X}$ be {a Zariski cover by an affine scheme}, and let $\mathcal{R}:=\widetilde{\mathcal{X}}\times_{\mathcal{X}}\widetilde{\mathcal{X}}$. Note that $\mathcal{R}$ is an affine scheme, the two projections $\mathcal{R} \rightarrow \widetilde{\mathcal{X}}$ are Zariski covers, and the embedding $\mathcal{R} \rightarrow \widetilde{\mathcal{X}} \times \widetilde{\mathcal{X}}$ makes $\mathcal{R}$ into an equivalence relation. By Lemmas \ref{lem:affine.case} and \ref{lem:generic.cover}, there is an open set $U \subset S$, $U$-schemes $R,\widetilde{X}$, and a monomorphism $R \rightarrow \widetilde{X}\times \widetilde{X}$ such that the two projections $R \rightarrow \widetilde{X}$ are Zariski covers, $R$ is an equivalence relation on $\widetilde{X}$, and the map $R_\eta \rightarrow \widetilde{X}_\eta\times\widetilde{X}_\eta$ is isomorphic to $\mathcal{R} \rightarrow \widetilde{\mathcal{X}}\times \widetilde{\mathcal{X}}$. Let $X$ be the gluing of $\widetilde{X}$ along $R$. Then $\widetilde{X}$ is a scheme and $X_\eta$ is isomorphic to $\mathcal{X}$.
\item {We prove that i}f $\mathcal{X}$ is an arbitrary scheme over $\eta$, then there is an open $U \subset S$ and a scheme $X$ over $U$ such that $X_\eta=\mathcal{X}$:\\
The proof is similar. The only difference is that, in this case, $\mathcal{R}$ is separated by not necessarily affine. Instead of using Lemma \ref{lem:affine.case}, we use the previous step and the fully faithfulness.
\item {We prove that i}f $\mathcal{X}$ is an algebraic space over $\eta$, then there is an open $U \subset S$ and an algebraic space $X$ over $U$ such that $X_\eta=\mathcal{X}$:\\
The proof is similar to the previous step, replacing Zariski covers by etale covers.
\end{enumerate}
\end{description}
\end{proof}

\subsection{Group algebraic spaces}

\begin{proposition}[{\cite[5.1.1]{Beh} \label{prop:group.algebraic.space} \footnote{The proof there uses implicitly \cite[60.13.2]{SP}}}]
Let $X$ be a scheme and $G$ be a group algebraic space over it.
Then there exists a stratification $X=\bigcup X_i$
such that $G|_{X^{{red}}_i}$ is  a group scheme. Here $X^{{red}}_i$ is the reduction of $X_i$.
\end{proposition}


\subsubsection{Stratification of QCA stacks}

\begin{proposition}\label{prop:strat.qca.1}
Let $X$ be a scheme and $G$ be a group {scheme} over it. Assume that $\fX:=(BG)_{fppf}$ is a QCA stack. Then there exists a stratification $X=\bigcup X_i$ such that
$G|_{X^{{red}}_i}$  can be embedded as a  closed subgroup in $GL_n\times_{spec \bZ} X^{{red}}_i$ for some $n$.
\end{proposition}

\begin{proof}
Without loss of generality, we may assume that $X$ is reduced, irreducible and affine.
 By Noetherian induction, it is enough to prove that there exist open $U \subset X$ and an integer $n$ such that $G|_{U}$ can be embedded as a closed subgroup in $GL_n\times_{spec \bZ} U$.
Let $\eta$ be the generic point of $X$.
Denote the composition $\eta \to X \to \fX$ by $x$. By definition, the group $Aut(x)$ is $G_\eta$. Therefore $G_\eta$ is linear.
Thus we have a closed embeding  morphism  $G_{\eta}\to GL_n\times_{spec \bZ} \eta$.
By Proposition \ref{prop:gen.pt}, we can extend this embeding to an embeding $\phi:G_V \to GL_n\times_{spec \bZ} V$ for some affine open $V \subset X$, as required.


\end{proof}

\subsubsection{Quotients of group schemes}

\begin{definition}[Quotient of group schemes]
Let $H \subset G$ be an embedding of group schemes and let $Y$ be a scheme. A morphism of schemes $G \to Y$ is a quotient iff
the map $G \times H \to G\times_Y G$ given by $(g,h)\mapsto (g,gh)$ is an  isomorphism.
\end{definition}


\begin{lemma} \label{lem:base.change.quotient} Let $X$ be a scheme, let $H \subset G$ be group schemes over $X$ such that the quotient $p:G \rightarrow G/H$ exists, and let $f:H \rightarrow G$ be an $H$-equivariant map. Denote the structure map of $H$ by $s_H:H \rightarrow X$ and the identity map of $H$ by $1_H:X\rightarrow H$. Let $\nu := p \circ f \circ 1_H:X \rightarrow G$. Then the diagram
\[
\xymatrix{H \ar@{->}[r]^f \ar@{->}[d]^{s_H} & G \ar@{->}[d]^p\\ X \ar@{->}[r]^{\nu}& G/H}
\]
is cartesian.
\end{lemma}
\begin{proof}
Without loss of generality we can assume that $f$ is the group embedding.
We get the desired square by composing the following two:
$$
\xymatrix{G\times_X H \ar@{->}[rr]^{(g,h)\mapsto gh} \ar@{->}[d]^{(g,h)\mapsto g} && G \ar@{->}[d]^p\\ G \ar@{->}[rr]^{p}&& G/H}
$$
and
$$
\xymatrix{H \ar@{->}[rr]^{h\mapsto (1,h)} \ar@{->}[d]^{(g,h)\mapsto g} && G\times_X H \ar@{->}[d]^{(g,h)\mapsto g}\\ X \ar@{->}[rr]^{1_G}&& G}
$$
\end{proof}

\begin{proposition}\label{prop:strat.qca}
Let $X$ be a scheme, let $G$ be a smooth and affine group algebraic space over $X$, and let $H \subset G$ be a subgroup algebraic space over $X$. Then there exists a stratification $X=\bigcup X_i$ such that
$G|_{\bar X_i}$ and $H|_{\bar X_i}$ are group schemes and
the quotients $G|_{\bar X_i}/H|_{\bar X_i}$  exist. 
\end{proposition}
\begin{proof}
Without loss of generality, we may assume that $X$ is reduced, irreducible and affine.
Using Proposition \ref{prop:group.algebraic.space}, we can assume that $G$ is  a group scheme. By Noetherian induction, it is enough to prove that there exists open $U \subset X$ such that the quotient $G|_{U}/H|_{U}$ exists. Let $\eta$ be the generic point of $X$. By \cite[Corollary 1.2]{Con}, the quotient $Y:=G|_{\eta}/H|_{\eta}$ exists. The assertion now follows from Proposition \ref{prop:gen.pt}.


\end{proof}

\end{document}